\documentclass[11pt]{amsart}
\usepackage{amsmath, amssymb}
\usepackage{amsfonts}
\usepackage{graphicx}
\usepackage{mathrsfs}
\usepackage[arrow,matrix,curve,cmtip,ps]{xy}
\usepackage{amsthm}
\usepackage{enumerate}
\usepackage{color}
\usepackage[colorlinks]{hyperref}
\usepackage{tikz}
\allowdisplaybreaks

\newtheorem{theorem}{Theorem}[section]
\newtheorem{lemma}[theorem]{Lemma}
\newtheorem{proposition}[theorem]{Proposition}
\newtheorem{corollary}[theorem]{Corollary}

\newtheorem*{theorem*}{Theorem}
\theoremstyle{remark}
\newtheorem{remark}[theorem]{Remark}
\newtheorem{definition}[theorem]{Definition}
\newtheorem{example}[theorem]{Example}

\numberwithin{equation}{section}

\begin{document}
\title[Prime and primitive Kumjian-Pask algebras]{Prime and primitive Kumjian-Pask algebras}

\author[M. Kashoul-Radjabzadeh, H. Larki, and A. Aminpour]{Maryam Kashoul-Radjabzadeh, Hossein Larki, Abdolmohammad Aminpour}

\address{Department of Mathematics\\
Faculty of Mathematical  Sciences and Computer\\
Shahid Chamran University of Ahvaz\\
Iran}
\email{m-rajabzadeh@phdstu.scu.ac.ir, h.larki@scu.ac.ir, aminpour@scu.ac.ir}


\date{\today}



\begin{abstract}
In this paper, prime as well as primitive Kumjian-Pask algebras $\mathrm{KP}_R(\Lambda)$ of a row-finite $k$-graph $\Lambda$ over a unital commutative ring $R$ are completely characterized in graph-theoretic and algebraic terms. By applying quotient $k$-graphs, these results describe prime and primitive graded basic ideals of Kumjian-Pask algebras. In particular, when $\Lambda$ is strongly aperiodic and $R$ is a field, all prime and primitive ideals of a Kumjian-Pask algebra $\mathrm{KP}_R(\Lambda)$ are determined.
\end{abstract}

\maketitle

\section{Introduction}

For giving a generalization and graphical version of higher dimensional Cuntz-Krieger algebras \cite{steg1}, Kumjian and Pask in \cite{kumj00} introduced higher rank graphs (or $k$-graphs) and their relative $C^*$-algebras. The class of higher rank graph $C^*$-algebras naturally includes graph $C^*$-algebras and every (directed) graph may be considered as a $1$-graph. Since then, there have been great efforts to investigate the structure of higher rank graph $C^*$-algebras. (See \cite{raeburn,Robert,kang} among others.)

A Leavitt path algebra $L_{R}(E)$ is an algebraic analogue of graph $C^*$-algebras $C^*(E)$, that was first introduced in \cite{abr1,ara} for a row-finite (directed) graph $E$ over a field $R$ and then extended in \cite{abr2} and \cite{tom1} for every graph $E$ and unital commutative ring $R$. Leavitt path algebras are so called because they also generalize the algebras $L(1,n)$ without invariant basis number studied by Leavitt in \cite{lea}. Despite the similarities in the definitions, there are some differences between a graph $C^*$-algebra $C^*(E)$ and the Leavitt path algebra $L_{R}(E)$ that one should take notice:
\begin{enumerate}
\item The coefficients of $C^*(E)$ belong to $\mathbb{C}$, but the coefficients of $L_{R}(E)$ belong to a unital commutative ring $R$.
\item Many similar results for these classes of algebras have been obtained independently by different techniques so that each one may not be applied for the others.
\item Even in the case $R=\mathbb{C}$, the algebras $C^*(E)$ and $L_{R}(E)$ may have different structural properties. For example, if $E$ is a graph with one vertex and one edge, $C^*$-algebra $C^*(E)$ (= C( $\mathbb{T}$)) is not prime where as $L_{\mathbb{C}}(E)$($=\mathbb{C}[x,x^{-1}]$) is a prime ring.  Thus in many cases one can not simultaneously investigate both graph $C^*$-algebras and Leavitt path algebras.
\end{enumerate}

Let $\Lambda$ be a row-finite $k$-graph without sources and $R$ be a unital commutative ring. In \cite{aranda} the authors introduced a $\mathbb{Z}^k$- graded algebra $\mathrm{KP}_{R}(\Lambda)$ which is called the Kumjian-Pask algebra. Kumjian-Pask algebras are the algebraic analogue of the $k$-graph $C^{\ast}$-algebras \cite{kumj00}. They then proved two kinds of uniqueness theorems, called the graded uniqueness theorem and the Cuntz- Krieger  uniqueness theorem and analyzed the ideal structure.

More recently in \cite{clark} Kumjian-Pask algebras were defined for arbitrary row-finite $k$-graphs being ``locally convex". It is shown that the desourcifying method could extend the results of \cite{aranda} for these algebras.

Recall that if $E$ is a directed graph and $R$ is a field, prime ideals as well as primitive ideals of $L_{R}(E)$ are described in \cite{ara3,ran}, via special subsets of the vertex set so-called maximal tails. First, the prime and primitive Leavitt-path algebras are determined by underlying graphs. So, prime graded ideals and primitive graded ideals of a Leavitt-path algebra are characterized by applying quotient graphs. Then it is shown that non graded ones are corresponding with prime ideals in $K[x,x^{-1}]$ and special families of maximal tails. However, if $R$ is not a field, the spaces of prime ideals and primitive ideals of $L_{R}(E)$ are more complicated.

The aim of this article is to investigate prime and primitive Kumjian-Pask algebras $\mathrm{KP}_{R}(\Lambda)$. We will allow the coefficient ring $R$ to be an arbitrary unital commutative ring rather than a field and also $\Lambda$ to be a locally convex row-finite $k$-graph. We first in Section $2$ review some basic definitions and results of higher rank graphs and their Kumjian-Pask algebras from \cite{aranda,clark}. Then in Section $3$, we provide equivalent conditions for the primeness and the primitivity of Kumjian-Pask algebra $\mathrm{KP}_{R}(\Lambda)$ of a row-finite $k$-graph $\Lambda$ with no sources. Finally, in Section 4, we apply the desourcifying method due to Farthing \cite{Far} to extend our main results to any row-finite locally convex $k$-graph with possible sources. In particular, the prime and primitive basic graded ideals of a Kumjian-Pask algebra $\mathrm{KP}_{R}(\Lambda)$ are described. Notice that when $R$ is a field and $\Lambda$ is strongly aperiodic, these results determine all prime and primitive ideals of $\mathrm{KP}_{R}(\Lambda)$.


\section{Preliminaries}

In this section, we recall some preliminaries about higher rank graphs and Kumjian-Pask algebras which will be needed in the next sections. We refer the reader to \cite{kumj00,raeburn,aranda,clark} for more details.

We denote the set of all natural numbers including zero by $\Bbb N$. Let $k$ be a positive integer. For $m,n\in \Bbb N^{k}$, we write $m\leq n$ if $m_{i}\leq n_{i}$ for every $1\leq i\leq k$ and write $m\vee n$ for the pointwise maximum of $m$ and $n$. We usually denote the zero element $(0,\ldots,0)\in \Bbb N^{k}$ by $0$.

\begin{definition}
For a positive integer $k$, a\emph{ k-graph} (or \emph{higher rank graph}) is a countable category $\Lambda=(\Lambda^{0},\Lambda,r,s)$ equipped with a functor $d:\Lambda\rightarrow\Bbb N^{k}$, called the $\emph{degree map}$, satisfying the following $\emph{factorization property}$:
for every $\lambda\in \Lambda$ and $m,n\in \Bbb N^{k}$ with $d(\lambda)=m+n$, there exist unique elements $\mu,\nu\in \Lambda$ such that $\lambda=\mu\nu$ and $d(\mu)=m$, $d(\nu)=n$.
\end{definition}

\begin{example}
For every ordinary directed graph $E=(E^0,E^1,r,s)$, the path category is naturally a $1$-graph. In this $1$-graph, $\Lambda^{0}$ is the set of vertices $E^{0}$ and the set of morphisms are the finite path space $\mathrm{Path}(E)$ of $E$. Also, the degree map $d:\mathrm{Path}(E)\rightarrow \Bbb N$ may be defined by $d(\mu)=|\mu|$ (the length of $\mu$).
\end{example}

\begin{example}
Let $\Omega_{k}^{0}:=\Bbb N^{k}$ and $\Omega_{k}:=\{(p,q)\in\Bbb N^{k}\times\Bbb N^{k}: p\leq q\}$. Define $r,s :\Omega_{k}\rightarrow\Omega_{k}^{0}$ and $d:\Omega_{k}\rightarrow \Bbb N^{k}$ by
$s(p,q):=q$, $r(p,q):=p$, $d(p,q):=q-p$. Then two pairs $(p,q)$ and $(r,s)$ are composable if $q=r$ and $(p,q)(r,s)=(p,s)$. By this definition $(\Omega_{k},r,s,d)$ is a $k$-graph. Moreover, for any $m\in \Bbb N^{k}$, we may define $\Omega_{k,m}^{0}:=\{p\in \Bbb N^{k} : p\leq m\}$ and $\Omega_{k,m}:=\{(p,q)\in\Bbb N^{k}\times\Bbb N^{k}: p\leq q\leq m\}$. With the same definition for $r,s$ and $d$ as above, $(\Omega_{k,m},r,s,d)$ would be a $k$-graph.
\end{example}

For $v,w\in \Lambda^{0}$ and $n\in \Bbb N^{k}$, we define $v\Lambda w:=\{\mu\in \Lambda : s(\mu)=w, r(\mu)=v\}$ and $\Lambda^{n}:=\{\mu\in \Lambda : d(\mu)=n\}$. Then by the factorization property we may identify $\Lambda^{0}$ with the objects of $\Lambda$ and so the elements of $\Lambda^{0}$ are called \emph{vertices}.

\begin{definition}
A $k$-graph $\Lambda$ is \emph{row-finite} if the set $v\Lambda^{n}=\{\mu\in\Lambda^{n}  : r(\mu)=v\}$ is finite for every $v\in\Lambda^{0}$ and $n\in \Bbb N^{k}$. Also, we say $\Lambda$ has \emph{no sources} if $v\Lambda^{n}\neq\emptyset$ for every $v\in\Lambda^{0}$ and $n\in\Bbb N^{k}$. A $k$-graph $\Lambda$ is called \emph{locally convex}, if for every $v\in \Lambda^{0}$ and distinct $i,j\in \{1,2,\ldots,k\}$, $\lambda\in v\Lambda^{e_{i}}$ and $\mu\in v\Lambda^{e_{j}}$, the sets $s(\lambda)\Lambda^{e_{j}}$ and $s(\mu)\Lambda^{e_{i}}$ are nonempty. Note that if a $k$-graph $\Lambda$ has no sources, it is trivially locally convex.
\end{definition}

For $n\in \Bbb N^{k}$, we set
$$\Lambda^{\leq n}:=\{\lambda\in \Lambda : d(\lambda)\leq n, ~\mathrm{and}~~ s(\lambda)\Lambda^{e_{i}}=\emptyset ~~\mathrm{whenever}~~ d(\lambda)+e_{i}\leq n\}.$$
If $\Lambda$ has no sources, we have $\Lambda^{\leq n}=\Lambda^{n}$.

\begin{definition}
Let $\Lambda$ be a row-finite locally convex $k$-graph and $m\in (\Bbb N\cup \{\infty\})^{k}$. A degree preserving functor $x: \Omega_{k,m}\rightarrow\Lambda$ is called a \emph{boundary path} of degree $m$ if for every $p\in \Bbb N^{k}$ and $i\in\{1,2,\ldots,k\}$, $p\leq m$ and $p_{i}=m_{i}$ imply that $x(p)\Lambda^{e_{i}}=\emptyset$. We denote the set of boundary paths by $\Lambda^{\leq \infty}$. If $x\in \Lambda^{\leq\infty}$, we usually call $x(0)\in \Lambda^0$ the range of $x$ and write it by $r(x)$. The degree of $x$ is denoted by $d(x)$.

For every $x\in \Lambda^{\leq\infty}$ and $n\in \Bbb N^{k}$ with $n\leq d(x)$, there is a boundary path $\sigma^{n}(x)$ of degree $d(x)-n$ which is defined by $\sigma^{n}(x)(p,q) :=x(p+n,q+n)$ for $p\leq q\leq d(x)-n$. Note that for every $p\leq d(x)$, we have $x(0,p)\sigma^p(x)=x$. Also, if $\lambda\in\Lambda x(0)$, there is a unique boundary path $\lambda x:\Omega_{k,m+d(\lambda)}\rightarrow\Lambda$ such that $\lambda x (0,d(\lambda))=\lambda$, $(\lambda x)(d(\lambda), d(\lambda) +p)=x(0,p)$. If $\Lambda$ has no sources, every boundary path is an infinite path; so, in this case, we denote $\Lambda^\infty=\Lambda^{\leq \infty}$.
\end{definition}

\begin{definition}
Let $\Lambda$ be a row-finite $k$-graph. we say that $\Lambda$ is \emph{aperiodic} if for every $v\in \Lambda^{0}$ and each $m\neq n\in \Bbb N^{k}$, there is $x\in \Lambda^{\leq\infty}$ such that either $m-m\wedge d(x) \neq n-{n\wedge d(x)}$ or $\sigma^{m\wedge d(x)}(x)\neq\sigma^{n\wedge d(x)}(x)$. We say $\Lambda$ is $\emph{periodic}$ if is not aperiodic.
\end{definition}

For a row-finite $k$-graph $\Lambda$ with no sources, every boundary path is an infinite path. So, for every $m,n\in\Bbb N^{k} $ and $x\in \Lambda^{\leq\infty}$, we have $m\wedge d(x)=m$, $n\wedge d(x)=n$. Hence, when $\Lambda$ has no sources, the definition of aperiodicity presented above is equivalent to the same definition given in \cite {aranda}.

Let $\Lambda$ be a $k$-graph and $\Lambda^{\neq 0}=\{\lambda\in \Lambda : d(\lambda)\neq 0\}$. For every $\lambda\in \Lambda$, we define a $\emph{ghost path}$ $\lambda^{\ast}$ and for $v\in \Lambda^{0}$, we define $v^{\ast}:=v$. The set of ghost paths are denoted by $G(\Lambda)$.
We extend $d,r,s$ to $G(\Lambda)$ by $$d(\lambda^{\ast})=-d(\lambda) ~~, ~~ r(\lambda^{\ast})=s(\lambda) ~~ ,  ~~s(\lambda^{\ast})=r(\lambda).$$
Also the composition on $G(\Lambda)$ is defined by $\lambda^{\ast}\mu^{\ast}=(\mu\lambda)^{\ast}$ for $\mu ,\lambda\in \Lambda^{\neq 0}$.

\begin{definition}
Let $\Lambda$ be a row-finite locally convex $k$-graph and $R$ be a unital commutative ring. A $\emph{Kumjian-Pask}$ $\Lambda$-$\emph{family}$ $(P,S)$ in an $R$-algebra $A$ consists of two functions $P:\Lambda^{0}\rightarrow A$ and $S:\Lambda^{\neq 0}\cup G(\Lambda^{\neq 0})\rightarrow A$ satisfying the following conditions.
\begin{enumerate}[(KP1)]
\item $\{P_{v} : v\in \Lambda^{0}\}$ is a family of mutually orthogonal idempotents.
\item for all $\lambda , \mu\in\Lambda^{\neq 0}$ with $r(\mu)=s(\lambda)$, we have $S_{\lambda}S_{\mu}=S_{\lambda\mu}$, $S_{\mu^{\ast}}S_{\lambda^{\ast}}=S_{(\lambda\mu)^{\ast}}$ and
        \begin{align*}
         &P_{r(\lambda)}S_{\lambda}=S_{\lambda}P_{s(\lambda)}=S_{\lambda},\\ &P_{s(\lambda)}S_{\lambda^{\ast}}=S_{\lambda^{\ast}}P_{r(\lambda)}=S_{\lambda^{\ast}}.
        \end{align*}
\item for all $n\in \mathbb{N}\setminus\{0\}$ and $\lambda,\mu \in \Lambda^{\leq n}$, we have $S_{\lambda^{\ast}}S_{\mu}=\delta_{\lambda,\mu}P_{s(\lambda)}$.
\item for all $v\in \Lambda^0$ and $n\in \mathbb{N}^k\setminus \{0\}$,
$$P_{v}=\sum_{\lambda\in v\Lambda^{\leq n}}S_{\lambda}S_{\lambda^{\ast}}.$$
\end{enumerate}
\end{definition}

For every locally convex $k$-graph $\Lambda$, \cite [Theorem 3.7]{clark} implies that there is an $R$-algebra $\mathrm{KP}_{R}(\Lambda)$ generated by a Kumjian-Pask $\Lambda$-family $(P,S)$ such that if $(Q,T)$ is a Kumjian-Pask $\Lambda$-family in an $R$-algebra $A$, there exists an $R$-algebra homomorphism $\pi_{Q,T}: \mathrm{KP}_{R}(\Lambda)\rightarrow \emph{A}$ such that $\pi_{Q,T}(P_{v})=Q_{v}$, $\pi_{Q,T}(S_{\lambda})=T_{\lambda}$, and $\pi_{Q,T}(S_{\lambda^{\ast}})=T_{\lambda^{\ast}}$. The $R$-algebra $\mathrm{KP}_{R}(\Lambda)$ is called the \emph{Kumjian-Pask algebra of $\Lambda$} with coefficients in $R$. A calculation shows that
\begin{equation}\label{equ2.1}
\mathrm{KP}_R(\Lambda)= \mathrm{span}_R \{s_{\alpha}s_{\beta^{\ast}} : \alpha,\beta\in \Lambda , s(\alpha)=s(\beta)\}
\end{equation}
(cf. \cite[Theorem 3.4]{aranda} and \cite[Corolary 3.4]{clark}).

Recall that a ring $\emph{R}$ is called $\mathbb{Z}^{k}$-$\emph{graded}$ if there is a collection of additive subgroups $\{R_n\}_{n\in\mathbb{Z}^{k}}$ of $\emph{R}$ such that $\emph{R}_{n_{1}}\emph{R}_{n_{2}}\subseteq R_{n_{1}+n_{2}}$ and every nonzero element $a\in \emph{R}$ can be written uniquely as finite sum $\sum_{n\in G} a_{n}$ of nonzero elements $a_{n}\in \emph{R}_{n}$. The subgroup $\emph{R}_n$ is said the \emph{homogeneous component of $\emph{R}$ of degree $n$}. In this case, an ideal $I$ of $R$ is called $\emph{graded }$ if $\{I\cap \emph{R}_{n} : n\in \mathbb{Z}^{k}\}$ is a grading of $I$. Furthermore, if $\phi :R\rightarrow S$ is a ring homomorphism between graded rings, $\phi$ is called $\emph{graded}$ if $\phi(\emph{R}_{_{n}})\subseteq \emph{S}_{n}$ for all $n\in \mathbb{Z}^{k}$.
Note that the kernel of a graded homomorphism is always a graded ideal. If $I$ is a graded ideal of ring $\emph{R}$, then the quotient $\emph{R}/I$ is naturally graded with homogeneous components $\{\emph{R}_n+I\}_{n\in\mathbb{Z}^{k}}$ and the quotient map $\emph{R}\rightarrow \emph{R}/I$ is a graded homomorphism. We see from \cite[Theorem 3.4]{aranda} and \cite[Theorem 3.7]{clark} that every Kumjian-Pask algebra $\mathrm{KP}_{\emph{R}}(\Lambda)$ is $\mathbb{Z}^k$-graded by
$$\mathrm{KP}_R(\Lambda)_{n}= \mathrm{span}_R \{s_{\alpha}s_{\beta^{\ast}} : \alpha,\beta\in \Lambda ,~~~~ d(\alpha)-d(\beta)=n\}$$
for $n\in \mathbb{Z}^k$.


\section{Prime and primitive Kumjian-Pask algebras}

Let $\Lambda$ be a row-finite $k$-graph with no sources and $\emph{R}$ be a unital commutative ring. In this section, we give some conditions for $\Lambda$ and $\emph{R}$ such that $\mathrm{KP}_{\emph{R}}(\Lambda)$ is a prime ring as well as primitive ring. We will generalize our results in Section $4$ for row-finite locally convex $k$-graphs.

Recall that an ideal $I$ of $\mathrm{KP}_{R}(\Lambda)$ is called \emph{basic} if $rp_{v}\in I$ with $r\in R\setminus \{0\}$ implies $p_{v}\in I$. An ideal $I$ of ring $R$ is called \emph{prime} if for each pair of ideals $I_{1},I_{2}$ of $R$ with $I_{1}I_{2}\subseteq I$, at least one of them is contained in $I$. Also, we say a ring $R$ is \emph{prime} if the zero ideal of $R$ is prime. Recall from \cite[Proposition II.1.4]{nas} that in a graded algebra, a graded ideal $I$ is prime if and only if for any graded ideals $I_1,I_2$ with $I_1 I_2\subseteq I$, we have $I_1\subseteq I$ or $I_2\subseteq I$.  We will use this fact in the proof of Theorem \ref{thm3.3}.

To give some necessary and sufficient conditions for the primeness of $\mathrm{KP}_{R}(\Lambda)$ in Theorem \ref{thm3.3}, we need the following lemma.

\begin{lemma}\label{vlemma1}
Let $\Lambda$ be a row-finite $k$-graph with no sources and $R$ be a unital commutative ring. If $I$ is a nonzero graded ideal of $\mathrm{KP}_{R}(\Lambda)$, then there exists $rp_{v}\in I$ for some $v\in \Lambda^{0}$ and $r\in R\setminus\{0\}$.
\end{lemma}

\begin{proof}
Suppose that $I$ is a graded ideal of $\mathrm{KP}_{\emph{R}}(\Lambda)$. Take an element $x\in I_{n}=I\cap \mathrm{KP}_{R}(\Lambda)_{n}$ for some $n\in \mathbb{Z}^k$. Since $x\in \mathrm{KP}_{R}(\Lambda)_{n}$, we can write $x=\sum_{i=1}^m r_{i}s_{\alpha_{i}}s_{{\beta_{i}}^\ast}$ such that $ d(\alpha_{i})-d(\beta_{i})=n$, where $\alpha_{i}, \beta_{i}\in \Lambda$ and $r_i\in R\setminus\{0\}$. Set $t=\vee_{1\leq i\leq m} d(\beta_{i})$. As $\Lambda$ has no sources, for each $1\leq i\leq m$ apply $(\mathrm{KP4})$ with $t_{i}=t-d(\beta_{i})$ to have
$$s_{\alpha_{i}}s_{{\beta_{i}}^\ast} =s_{\alpha_{i}} p_{s(\alpha_{i})}s_{{\beta_{i}}^\ast} =\sum_{\lambda\in s(\alpha_{i})\Lambda^{t_{i}}} s_{(\alpha_{i}\lambda)}s_{({\beta_{i}\lambda})^{\ast}}.$$
It is clear that $d(\alpha_{i}\lambda)- d(\beta_{i}\lambda) =n$. Hence, we can write
$$x=\sum_{i=1}^{m} r_{i}s_{\alpha_{i}}s_{{\beta_{i}}^\ast}=\sum_{i=1}^{m} \sum _{\lambda\in s(\alpha_{i})\Lambda^{t_{i}}} r_{i} s_{(\alpha_{i}\lambda)}s_{({\beta_{i}\lambda})^{\ast}},$$
where the degree of each ${\beta_{i}\lambda}$ is $t$. This yields that all ${\alpha_{i}\lambda}$ also have the same degree $n-t$ and by (KP3) we have
$$s_{(\alpha_{1}\lambda)^*} x s_{{\beta_{1}\lambda}} = r_{1}p_{s({\alpha_{1}\lambda})}\in I$$
for every $\lambda\in s(\alpha_1)\Lambda^{t_1}$.
\end{proof}

Let $\Lambda$ be a row-finite locally convex $k$-graph. We define a relation on $\Lambda^{0}$ by setting $v\leq w$ if there is some $\lambda\in v\Lambda w$. A subset $H$ of $\Lambda^{0}$ is $\emph{hereditary}$ if $v\in H$ and $v\Lambda w\neq\emptyset$ imply $w\in H$. Also, $H$ is$\emph{ saturated}$ if $\{s(\lambda) : \lambda\in v\Lambda^{\leq{e_{i}}}\}\subseteq H$ for some $1\leq i\leq k$ implies $v\in H$. So, for a row-finite $k$-graph $\Lambda$ with no sources, the definition of saturation is equivalent to: $s(v\Lambda^{n})\subseteq H$ for some $n\in \Bbb N^{k}$ implies $v\in H$. The saturation of $H$, denoted by $\overline{H}$, is the smallest saturated subset of $\Lambda^{0}$ containing $H$.
Recall that for a saturated and hereditary subset $H$ of $\Lambda^{0}$,
$$I_{H}=\mathrm{span}_{R}\{s_{\mu}s_{\lambda^*}: s(\mu)=s(\lambda)\in H\}$$
is a basic and graded ideal of $\mathrm{KP}_{R}(\Lambda)$ \cite[Theorem 9.4]{clark}. Also, for an ideal $I$ of $\mathrm{KP}_{R}(\Lambda)$, we define $H_{I}:=\{v\in \Lambda^{0} : p_{v}\in I\}$. From \cite[Lemma 5.2]{aranda}, $H_{I}$ is a saturated and hereditary subset of $\Lambda^{0}$.

\begin{definition}
Let $\Lambda$ be a row-finite locally convex $k$-graph. A nonempty subset \emph{M} of $\Lambda^{0}$ is called a \emph{maximal tail} if \emph{M} satisfies in the following conditions.
\begin{enumerate}[(MT1)]
\item If $w\in\emph{M}$ and $v\in\Lambda^{0}$ with $v\Lambda w\neq\emptyset$, then $v\in\emph{M}$.
\item If $v\in\emph{M}$, for every $1\leq i\leq k$ there exists $\lambda\in v\Lambda^{\leq e_{i}}$ such that $s(\lambda)\in \emph{M}$.
\item For $v_{1},v_{2}\in \emph{M}$, there exists $w\in\emph{M}$ such that $v_{1}\Lambda w\neq\emptyset$ and $v_{2}\Lambda w\neq\emptyset$.
\end{enumerate}
\end{definition}

Note that a subset $H\subseteq \Lambda^{0}$ is hereditary and saturated if and only if $\Lambda^0\setminus H$ satisfies Conditions (MT1) and (MT2). In the following theorem, we characterize the prime Kumjian-Pask algebras when the underlying $k$-graphs have no sources. It will be generalized in Theorem \ref{thm4.5} for all locally convex row-finite $k$-graphs.

\begin{theorem}\label{thm3.3}
Let $\Lambda$ be a row-finite $k$-graph with no sources and ${R}$ be a unital commutative ring. Then the following are equivalent.
\begin{enumerate}
\item $\mathrm{KP}_{{R}}(\Lambda)$ is a prime ring.
\item ${R}$ is an ID (Integral Domain) and $\Lambda^{0}$ satisfies Condition (MT3).
\item ${R}$ is an ID and $H_{I}\cap H_{J}\neq \emptyset$ for every nonzero basic graded ideals $I,J$ of $\mathrm{KP}_{R}(\Lambda)$.
\end{enumerate}
\end{theorem}

\begin{proof}
{\bf 1 $\Rightarrow$ 2:} Suppose that $\mathrm{KP}_{R}(\Lambda)$ is a prime ring and $v,w\in \Lambda^{0}$. Then $\mathrm{KP}_{R}(\Lambda)p_{v}\mathrm{KP}_{R}(\Lambda)$ and $\mathrm{KP}_{R}(\Lambda)p_{w}\mathrm{KP}_{R}(\Lambda)$ are two nonzero ideals. By the primeness of $\mathrm{KP}_{R}(\Lambda)$, $\mathrm{KP}_{R}(\Lambda)p_{w}\mathrm{KP}_{R}(\Lambda)p_{v}\mathrm{KP}_{R}(\Lambda)$ is also nonzero and so is the corner $p_{w}\mathrm{KP}_{R}(\Lambda)p_{v}$. By relation (\ref{equ2.1}), there are $\alpha,\beta\in \Lambda$ such that $p_{w}s_{\alpha}s_{\beta^{\ast}}p_{v}\neq 0$ and $s(\alpha)=s(\beta)=z$. Hence $\alpha\in w\Lambda z$ and $\beta\in v\Lambda z$ which means $\Lambda^{0}$ satisfies Condition (MT3).
Now we show that $R$ is an ID. By contradiction, if there exist nonzero elements $r_{1},r_{2}\in R$ such that $r_{1}r_{2}=0$, then $r_{1}\mathrm{KP}_{R}(\Lambda)$ and $r_{2}\mathrm{KP}_{R}(\Lambda)$ are two nonzero ideals of $\mathrm{KP}_{R}(\Lambda)$. But
$r_{1}\mathrm{KP_{R}(\Lambda)} r_{2}\mathrm{KP}_{R}(\Lambda)= r_{1}r_{2}\mathrm{KP}_{R}(\Lambda)=\{0\}$, which contradicts the primeness of $\mathrm{KP}_{R}(\Lambda)$.

{\bf 2 $\Rightarrow$ 3:} Let $I$ and $J$ be two nonzero basic graded ideals of $\mathrm{KP}_{R}(\Lambda)$. By Lemma \ref{vlemma1}, there are $v,w\in \Lambda^{0}$ such that $p_{v} \in I$ and $p_{w}\in J$. Condition (MT3) implies that there is $z\in \Lambda^{0}$ such that $v\Lambda z \neq \emptyset$ and $w\Lambda z\neq\emptyset$ and so $p_{z}\in I\cap J$ and $H_{I}\cap H_{J}\neq \emptyset$.

{\bf 3 $\Rightarrow$ 1:} Let $\emph{R}$ be an ID and $H_{I}\cap H_{J} \neq\emptyset$ for every nonzero basic graded ideals $I,J$ of $\mathrm{KP}_{R}(\Lambda)$. We show that the zero ideal of $\mathrm{KP}_{R}(\Lambda)$ is prime. Since the zero ideal is graded, it is sufficient to show that $IJ\neq \{0\}$ for every nonzero graded ideals $I,J$ of $\mathrm{KP}_{R}(\Lambda)$. If $I$ and $J$ are such ideals, by Lemma \ref{vlemma1} there are $v_{1},v_{2}\in \Lambda^{0}$ and $r_{1},r_{2}\in R\setminus\{0\}$ such that $r_{1}p_{v_{1}}\in I$ and $r_{2}p_{v_{2}}\in J$. If we set $H_{i}=\{z\in \Lambda^{0}: v_{i}\Lambda z\neq\emptyset\}$ for $i\in \{1,2\}$, it is clear that $H_{1}$ and $H_{2}$ are two hereditary subsets of $\Lambda^{0}$. Let $I'=I_{\overline{H_{1}}}$ and $J'=I_{\overline{H_{2}}}$. Then $p_{v_{1}}\in I^{\prime}$, $p_{v_{2}}\in J^{\prime}$ and we have $r_{1}I^{\prime}\subseteq I$ and $r_{2}J^{\prime}\subseteq J$. Since $I^{\prime},J^{\prime}$ are basic graded ideals of $\mathrm{KP}_{R}(\Lambda)$, from statement $(3)$ we get $\overline{H_{1}}\cap\overline{H_{2}}\neq \emptyset$. Thus, there is $z_{0}\in \overline{H_{1}}\cap\overline{H_{2}}$ and since \emph{R} is an ID, we have $r_{1}r_{2}\neq 0$, $r_{1}r_{2}p_{z_{0}}\in IJ$, and so $IJ\neq \{0\}$.
\end{proof}

Now we consider primitive Kumjian-Pask algebras. Recall that a ring $R$ is said \emph{left primitive} (\emph{right primitive}) if it has a faithful simple left (right) \emph{R}-module. Since a Kumjian-Pask algebra is left primitive if and only if it is right primitive, in this case, we simply say to be \emph{primitive}. Note that every primitive ring is prime. Also, every commutative primitive ring is a field.

\begin{lemma}[{\cite[Lemmas 2.1 and 2.2]{lan}}]\label{lemma3}
Let $R$ be a field and $R_{1}$ be a prime $R$-algebra. Then there exists a prime unital $R$-algebra $R_{2}$ which $R_{1}$ embeds in $R_{2}$ as an ideal. Furthermore, $R_{2}$ is primitive if and only if $R_{1}$ is primitive.
\end{lemma}

\begin{lemma}[{\cite [Theorem 1]{forman}}]\label{lemma4}
A unital ring $R$ is left primitive if and only if there is a left ideal $M\neq R$ of $R$ such that for every nonzero two sided ideal $I$ of $R$, we have $M+I=R$.
\end{lemma}

We use the infinite-path representation of a $k$-graph defined in \cite[Section 3]{aranda} to prove Theorem \ref{thm3.8} below.

\begin{definition}[\cite{aranda}]
Let $\Lambda$ be a row-finite $k$-graph with no sources and $\emph{R}$ be a unital commutative ring. For $v\in \Lambda^{0}$ and $\mu,\nu\in \Lambda^{\neq 0}$, we define the maps $f_{v} , f_{\lambda} , f_{\mu^{\ast}} : \Lambda^{\infty}\rightarrow \Bbb F_{R}(\Lambda^{\infty})$ by
\begin{align*}
f_{v}(x) &= \left\{
\begin{array}{ll}
\ x &\mathrm{ if}~~~ x(0)=v \\
\ 0  & \mathrm{otherwise},\\
\end{array}
\right.\\
f_{\lambda}(x) &= \left\{
\begin{array}{ll}
\ \lambda x &\mathrm{ if}~~~x(0)=s(\lambda) \\
\ 0  & \mathrm{otherwise , ~ and}\\
\end{array}
\right.\\
f_{\mu^{\ast}}(x) &= \left\{
\begin{array}{ll}
\ x(d(\mu),\infty)  &\mathrm{ if}~~~ x(0,d(\mu))=\mu \\
\ 0   & \mathrm{otherwise,}\\
\end{array}
\right.
\end{align*}
where $\Bbb F_{R}(\Lambda^{\infty})$ is the free module with basis the infinite path space. By the universal property of free modules, there are nonzero endomorphisms $Q_{v} , T_{\lambda} , T_{\mu^{\ast}} : \Bbb F_{R}(\Lambda^{\infty})\rightarrow \Bbb F_{R}(\Lambda^{\infty})$ extending $f_{v} , f_{\lambda} , f_{\mu^{\ast}}$. In \cite{aranda}, it is shown that $(Q,T)$ is a Kumjian-Pask $\Lambda$-family in $\mathrm{End}(\Bbb F_{R}(\Lambda^{\infty}))$. So there is an $R$-algebra homomorphism $\pi_{Q,T} : \mathrm{KP}_{R}(\Lambda) \rightarrow \mathrm{End}(\Bbb F_{R}(\Lambda^{\infty}))$ such that $\pi_{Q,T}(p_{v})=Q_{v}$, $\pi_{Q,T}(s_{\lambda})=T_{\lambda}$, and $\pi_{Q,T}(s_{\mu^{\ast}})=T_{\mu^{\ast}}$. The homomorphism $\pi_{Q,T}$ is called the \emph{infinite-path representation of $\mathrm{KP}_{R}(\Lambda)$}.
\end{definition}

Recall from Corollary 4.10 and Lemma 5.9 of \cite{aranda} that the infinite-path representation $\pi_{Q,T}$ of $\mathrm{KP}_R(\Lambda)$ is faithful if and only if $\Lambda$ is aperiodic.

\begin{theorem}\label{thm3.8}
Let $\Lambda$ be a row-finite $k$-graph with no sources and $R$ be a unital commutative ring. Then $\mathrm{KP}_{R}(\Lambda)$ is primitive if and only if
\begin{enumerate}
\item $\Lambda^{0}$ satisfies Condition (MT3),
\item $\Lambda$ is aperiodic, and
\item $R$ is a field.
\end{enumerate}
\end{theorem}

\begin{proof}
First assume that the above three conditions hold. By Theorem \ref{thm3.3}, $\mathrm{KP}_{R}(\Lambda)$ is a prime ring. Then Lemma \ref{lemma3} implies that there exists a prime unital ring $R_{2}$ such that $\mathrm{KP}_{R}(\Lambda)$ embeds in $R_{2}$ as an ideal and primitivity of them are equivalent. So, it suffices to prove $R_2$ is a primitive ring. Taking an arbitrary vertex $v\in \Lambda^{0}$, suppose that $H=\{w\in \Lambda^{0}: v\leq w\}$ and write $H=\{v_{1},v_{2},\ldots\}$. We claim that there exists a sequence $\{\lambda_{i}\}_{i=1}^{\infty}$ of paths in $\Lambda$ such that for every $i\in \Bbb N$, $\lambda_{i+1}=\lambda_{i}\mu_{i}$ for some path $\mu_{i}\in \Lambda$ and also $v_{i}\leq s(\lambda_{i})$. For this, set $\lambda_{1}=v_{1}$. Clearly, $\lambda_{1}$ satisfies the properties. Assume that there are $\lambda_{1},\ldots, \lambda_{n}$ with the indicated properties. By Condition (MT3), there is $u\in \Lambda^{0}$ such that $v_{n+1}\leq u$ and $s(\lambda_{n})\leq u$ and so there is $\mu\in \Lambda$ with $s(\mu)=u$, $r(\mu)=s(\lambda_{n})$. If we set $\lambda_{n+1}:=\lambda_{n}\mu$, then $\lambda_{n+1}$ has the desired properties. Hence, the claim holds. Note that since each $\lambda_{i}$ is a subpath of $\lambda_{i+1}$, for every $n> i$, we have
$$s_{\lambda_{i}}s_{{\lambda_{i}}^{\ast}}s_{\lambda_{n}}s_{{\lambda_{n}}^{\ast}}= s_{\lambda_{n}}s_{{\lambda_{n}}^{\ast}}.$$

Define $M= \sum_{i=1}^{\infty} R_{2}(1-s_{\lambda_{i}}s_{{\lambda_{i}}^{\ast}})$ that is a left ideal of $R_{2}$. We see that $M \neq R_{2}$. Indeed, if $M=R_{2}$, we have $1\in M$. So, there are $r_{1},\ldots,r_{m}\in R_{2}$ such that $1=\sum_{i=1}^{m} r_{i}(1-s_{\lambda_{i}}s_{{\lambda_{i}}^{\ast}})$ which follows
$$s_{\lambda_{m}}s_{{\lambda_{m}}^{\ast}}=\sum_{i=1}^{m} r_{i}(1-s_{\lambda_{i}}s_{{\lambda_{i}}^{\ast}})s_{\lambda_{m}}s_{{\lambda_{m}}^{\ast}}=0,$$
a contradiction. Therefore, $1\notin M$ and $M\neq R_{2}$.

Now suppose that $I$ is an arbitrary two sided ideal of $R_{2}$. We show that $M+I=R_{2}$. Since $R_{2}$ is prime and $\mathrm{KP}_{R}(\Lambda)$ is a two sided ideal of $R_{2}$, we have $I_{1}=I\cap \mathrm{KP}_{R}(\Lambda)$ is a nonzero two sided ideal of $\mathrm{KP}_{R}(\Lambda)$. Since $\Lambda$ is aperiodic, an application of the Cuntz-Krieger uniqueness theorem implies that $I_1$ contains a vertex idempotent $p_w$ (see \cite[Proposition 5.11]{aranda}). By Condition (MT3), there exists $z\in \Lambda^{0}$ such that $w\leq z$ and $v\leq z$. So, $z=v_{n}$ for some $n\geq 1$. This yields that $p_{s(\lambda_{n})}\in I$ and $s_{\lambda_{n}}s_{{\lambda_{n}}^{\ast}}\in I$. As $1=(1-s_{\lambda_{n}}s_{{\lambda_{n}}^{\ast}})+s_{\lambda_{n}}s_{{\lambda_{n}}^{\ast}}$, we get $1\in M+I$ and $M+I=R_{2}$. Therefore, the left ideal $M$ satisfies the conditions of Lemma \ref{lemma4} and so $R_{2}$ is primitive. By Lemma \ref{lemma3}, we conclude that $\mathrm{KP}_{R}(\Lambda)$ is primitive.

Conversely, suppose that $\mathrm{KP}_{R}(\Lambda)$ is  primitive. Therefore, $\mathrm{KP}_{R}(\Lambda)$ is a prime ring and $\Lambda^{0}$ satisfies Condition (MT3) by Theorem \ref{thm3.3}. We show that $\Lambda$ is aperiodic. Since $\mathrm{KP}_{R}(\Lambda)$ is a primitive ring, it has a faithful simple left $\mathrm{KP}_{R}(\Lambda)$-module $M$. Since $M$ is simple, there is a maximal left ideal $J$ of $\mathrm{KP}_{R}(\Lambda)$ such that $M\cong {\mathrm{KP}_{R}(\Lambda)}/{J}$ as modules. It follows $aM=0$ for every $a\in J$. As $M$ is faithful, we get $a=0$ whenever $a\in J$. So $J=0$, $M\cong \mathrm{KP}_{R}(\Lambda)$, and $\mathrm{KP}_{R}(\Lambda)$ is a simple $\mathrm{KP}_{R}(\Lambda)$-module. If $\Lambda$ is periodic, \cite[Lemma 5.9]{aranda} implies that the kernel of infinite-path representation $\pi_{Q,T}$ is a nonzero ideal of $\mathrm{KP}_{R}(\Lambda)$. Also, $\pi_{Q,T}$ contains no vertex idempotents \cite[Proposition 5.11]{aranda} and so, $\ker \pi_{Q,T}$ is a proper nonzero (two sided) ideal of $\mathrm{KP}_{R}(\Lambda)$. This contradicts the simplicity of $\mathrm{KP}_{R}(\Lambda)$, and hence, $\Lambda$ is aperiodic.

It remains to show $R$ is a field. If $M$ is a simple and faithful left $\mathrm{KP}_{R}(\Lambda)$-module, similar to above we have $M\cong \mathrm{KP}_{R}(\Lambda)$ as modules. If $R$ is not a field, there is a nonzero proper ideal $I$ of $R$. Then using \cite[Theorem 6.4(a)]{aranda}, $IM\cong I\mathrm{KP}_{R}(\Lambda)$ is a nonzero proper left $\mathrm{KP}_{R}(\Lambda)$-submodule of $M$. This contradicts the simplicity of $M$, and therefore, $R$ must be a field.
\end{proof}


\section{Locally convex $k$-graphs}

Let $\Lambda$ be a row-finite locally convex $k$-graph. In \cite[Theorem 7.4]{clark}, it is shown that there is a row-finite $k$-graph $\tilde{\Lambda}$ with no sources such that $\mathrm{KP}_{R}(\Lambda)$ and $\mathrm{KP}_{R}(\tilde{\Lambda})$ are Morita equivalent. In this section, we first review the construction of $\tilde{\Lambda}$ due to Farthing \cite{Far} and then generalize Theorems \ref{thm3.3} and \ref{thm3.8} for row-finite locally convex $k$-graphs with possible sources.

For a row-finite locally convex $k$-graph $\Lambda$, define the sets $V_{\Lambda}$ and $P_{\Lambda}$ as
\begin{align*}
V_{\Lambda}&:=\{ (x;m) : x\in \Lambda^{\leq \infty} , m\in \Bbb N^{k}\},\\
P_{\Lambda}&:=\{ (x;(m,n)) : x\in \Lambda^{\leq \infty} , m\leq n\in \Bbb N^{k}\}.
\end{align*}
If we define $(x;m)\approx (y;n)$ if and only if
\begin{enumerate}
\item[V1)] $x(m\wedge d(x))=y(n\wedge d(y))$ and
\item[V2)] $m-m\wedge d(x)=n-n\wedge d(y)$,
\end{enumerate}
then $\approx$ is an equivalence relation on $V_{\Lambda}$ and we denote the class of $(x;m)$ by $[x;m]$. Also, the relation $(x;(m,n))\thicksim (y;(p,q))$ if and only if
\begin{enumerate}
\item[P1)] $ x(m\wedge d(x), n\wedge d(x))= y(p\wedge d(y), q\wedge d(y))$,
\item[P2)] $ m-m\wedge d(x)=p-p\wedge d(y)$, and
\item[P3)] $ n-m=q-p$
\end{enumerate}
is an equivalence relation on $P_{\Lambda}$ and we denote the equivalence class of $(x;(m,n))$ by $[x;(m,n)]$.

Following \cite{Far}, there is a $k$-graph $\tilde{\Lambda}$ with $\tilde{\Lambda}^{0}:= V_{\Lambda}/ \approx$ and $\tilde{\Lambda}:= P_{\Lambda}/\sim$ such that
\begin{align*}
&r([x;(m,n)])=[x;m] \mathrm{~and~} s([x;(m,n)])=[x;n],\\
&[x;(m,n)] \circ [y;(p,q)]=[x(0,n\wedge d(x))\sigma^{p\wedge d(y)}; (m,n+q-p)],\\
&d([x;m])=0 \mathrm{~and~} d([x;(m,n)])=n-m.
\end{align*}
The $k$-graph $\tilde{\Lambda}=(\tilde{\Lambda}^{0}, \tilde{\Lambda}, r, s, d)$ contains no sources which is called the \emph{desourcification of $\Lambda$}. Recall from \cite[Corollary 7.5]{clark} that there is a surjective Morita context between $\mathrm{KP}_{R}(\tilde{\Lambda})$ and $\mathrm{KP}_{R}(\Lambda)$. Note that, by \cite[Theorem 7.4]{clark}, the primeness and primitivity are invariant under Morita contexts.

\begin{remark}
We may simply check that the paths $(\lambda x,(0,d(\lambda)))$ and $(\lambda y,(0,d(\lambda)))$ are equivalent under $\sim$ for any $\lambda\in \Lambda$ and $x,y\in s(\lambda)\Lambda^{\leq\infty}$. So, the map $\iota :\Lambda\rightarrow\tilde{\Lambda}$ satisfying $\iota(\lambda)=[\lambda x;(0,d(\lambda))]$ for any $x\in s(\lambda)\Lambda^{\leq\infty}$ is a well-defined map. Indeed, $\iota$ is an injective $k$-graph morphism and $\Lambda$ and $\iota(\Lambda)$ are isomorphic. Moreover, the map $\pi :\tilde{\Lambda}\rightarrow\iota(\Lambda)$ defined by $\pi[x;(m,n)]=[x;(m\wedge d(x),n\wedge d(x))]$ for $x\in \Lambda^{\leq\infty}$ and $m\leq n\in \Bbb N^{k}$ is a well-defined surjective $k$-graph morphism such that $\pi \circ\pi=\pi$ and $\pi \circ \iota=\iota$. See \cite[Section 6]{clark} for more details.
\end{remark}

\begin{lemma}[{\cite[Lemma 2.2]{Robert}}]\label{lemma7}
Let $\Lambda$ be a row-finite locally convex $k$-graph and $v\in \Lambda^{0}$. Suppose $x\in v\Lambda^{\leq\infty}$ and $p\in \Bbb N^{k}$ satisfying $p\wedge d(x)=0$. Then for any other $z\in v\Lambda^{\leq\infty}$, we have $p\wedge d(z)=0$ and $[x;(0,p)]=[z;(0,p)]$.
\end{lemma}

\begin{lemma}\label{lem4.3}
Let $\Lambda$ be a row-finite locally convex $k$-graph and $\tilde{\Lambda}$ be its desourcification. Then
\begin{enumerate}
\item $\tilde{\Lambda}$ is aperiodic if and only if $\Lambda$ is.
\item $\Lambda^{0}$ satisfies Condition (MT3) if and only if $\tilde{\Lambda}^{0}$ so does.
\end{enumerate}
\end{lemma}

\begin{proof}
The statement (1) is \cite[Proposition 3.6]{Robert}. For (2), suppose that $\Lambda^{0}$ satisfies Condition (MT3) and take arbitrary vertices $[x;m] , [y;n]\in \tilde{\Lambda}^{0}$. Without loss of generality, we may assume that $m\wedge d(x)=n\wedge d(y)=0$. Since $x(0),y(0)$ belong to $\Lambda^{0}$, there are $\lambda\in x(0)\Lambda$ and $\mu\in y(0)\Lambda$ such that $s(\lambda)=s(\mu)$. Take some $z\in s(\lambda)\Lambda^{\leq\infty}$. Then $\lambda z\in x(0)\Lambda^{\leq\infty}$, $\mu z\in y(0)\Lambda^{\leq\infty}$, and by Lemma \ref{lemma7} we have $[x;(0,m)]=[\lambda z; (0,m)]$ and  $[y;(0,n)]=[\mu z;(0,n)]$. If $t:=(m-d(\lambda))\vee (n-d(\mu))\vee 0$, then we have
\begin{align*}
r([\lambda z; (m, t+d(\lambda))])&=[\lambda z;m]=[x;m],\\
r([\mu z; (n, t+d(\mu))])&=[\mu z;n]=[y;n],\\
s([\lambda z; (m, t+d(\lambda))])&=[\lambda z;t+d(\lambda)]=[z;t], ~ \mathrm{and}\\
s([\mu z; (n, t+d(\mu))])&=[\mu z;t+d(\mu)]=[z;t].
\end{align*}
Therefore, $[\lambda z; (m, t+d(\lambda))])\in [x;m]\tilde{\Lambda}$ and $([\mu z; (n, t+d(\mu))])\in [y;n]\tilde{\Lambda}$ with the same source. This says that $\tilde{\Lambda}^{0}$ satisfies Condition (MT3), as desired.

Conversely, assume that $\tilde{\Lambda}^{0}$ satisfies Condition (MT3). If $v,w\in \Lambda^{0}$, there are $\lambda\in v\tilde{\Lambda}$ and $\mu\in w\tilde{\Lambda}$ such that $s(\lambda)=s(\mu)$. So, $\pi(\lambda) , \pi(\mu)\in \Lambda$ and we have
\begin{align*}
r(\pi(\lambda))&=\pi(r(\lambda))=\pi(v)=v,\\
r(\pi(\mu))&=\pi(r(\mu))=\pi(w)=w, ~ \mathrm{and} \\
s(\pi(\lambda))&=\pi(s(\lambda))=\pi(s(\mu))=s(\pi(\mu)).
\end{align*}
Thus $\pi(\lambda)\in v\Lambda , \pi(\mu)\in w\Lambda$ and $s(\pi(\lambda))=s(\pi(\mu))$. This implies that $\Lambda^{0}$ satisfies Condition (MT3).
\end{proof}

\begin{theorem}\label{thm4.5}
Let $\Lambda$ be a row-finite locally convex $k$-graph and $R$ be a unital commutative ring. Then $\mathrm{KP}_{R}(\Lambda)$ is prime if and only if R is an ID and $\Lambda^{0}$ satisfies Condition (MT3).
\end{theorem}

\begin{proof}
Let $\mathrm{KP}_{R}(\Lambda)$ be a prime ring. \cite[Corollary 7.5]{clark} yields that there is a surjective Morita context between  $\mathrm{KP}_{R}(\Lambda)$ and $\mathrm{KP}_{R}(\tilde{\Lambda})$. As the primeness is preserved under surjective Morita contexts, $\mathrm{KP}_{R}(\tilde{\Lambda})$ is also a prime ring. Since $\tilde{\Lambda}$ is a row-finite $k$-graph with no sources, Theorem \ref{thm3.3} implies that $R$ is an ID and $\tilde{\Lambda}^{0}$ satisfies Condition (MT3). Hence, $\Lambda^{0}$ satisfies Condition (MT3) either by Lemma \ref{lem4.3}.

For the converse, let $\Lambda^{0}$ satisfy Condition (MT3) and $R$ be an ID. As $\tilde{\Lambda}^{0}$ satisfies Condition (MT3) by Lemma \ref{lem4.3}, we see that $\mathrm{KP}_{R}(\tilde{\Lambda})$ is a prime ring by applying Theorem \ref{thm3.3}. Now the Morita equivalence between $\mathrm{KP}_{R}(\tilde{\Lambda})$ and $\mathrm{KP}_{R}(\Lambda)$ gives the result.
\end{proof}

Now, we consider prime graded ideals in the Kumjian-Pask algebra $\mathrm{KP}_{R}(\Lambda)$, when $\Lambda$ is a row-finite locally convex $k$-graph. As a basic graded ideal $I_H$ of $\mathrm{KP}_{R}(\Lambda)$ is prime if and only if the quotient $\mathrm{KP}_{R}(\Lambda)/I_H \cong \mathrm{KP}_{R}(\Lambda\setminus H)$ is prime, we may use Theorem \ref{thm4.5} to characterize prime basic graded ideals of the Kumjian-Pask algebra. Recall from \cite[Theorem 9.4]{clark} that every basic graded ideal of $\mathrm{KP}_R(\Lambda)$ is of the form $I_H$ for a saturated hereditary subset $H$ of $\Lambda^0$. Note that when $R$ is a field, every ideal of $\mathrm{KP}_R(\Lambda)$ is basic. We first state the following lemma.

\begin{lemma}\label{lem4.5}
Let $\Lambda$ be a row-finite locally convex $k$-graph and $R$ be a unital commutative ring. For a saturated and hereditary subset $H$ of $\Lambda^{0}$, $\Lambda\setminus H=(\Lambda^{0}\setminus H , s^{-1}(\Lambda^{0}\setminus H) , r,s)$ is a row-finite locally convex $k$-graph and $\mathrm{KP}_{R}(\Lambda\setminus H)$ is canonically isomorphic to $\mathrm{KP}_{R}(\Lambda)/ I_{H}$.
\end{lemma}

\begin{proof}
First, by \cite[Theorem 5.2]{raeburn}, $\Lambda\setminus H$ is a row-finite locally convex $k$-graph. Then, similar to the proof of \cite[Theorem 5.5]{aranda}, we may define an isomorphism between $\mathrm{KP}_{R}(\Lambda\setminus H)$ and $\mathrm{KP}_{R}(\Lambda)/ I_{H}$.
\end{proof}

We usually say a locally convex $k$-graph $\Lambda$ to be \emph{strongly aperiodic} if $\Lambda\setminus H$ is aperiodic for every hereditary saturated subset $H$ of $\Lambda^0$. When $\Lambda$ is strongly aperiodic with no sources, \cite[Corollary 5.7]{aranda} follows that every basic ideal of $\mathrm{KP}_R(\Lambda)$ is of the form $I_H$ for some saturated hereditary set $H$. Similarly, we may use \cite[Theorem 9.4]{clark} to have the same result for $\mathrm{KP}_R(\Lambda)$ when $\Lambda$ is a locally convex $k$-graph.

\begin{proposition}\label{prop4.6}
Let $\Lambda$ be a row-finite locally convex $k$-graph. A basic graded ideal $I_{H}$ of $\mathrm{KP}_R(\Lambda)$ is prime if and only if $\Lambda^{0}\setminus H$ is a maximal tail and $R$ is an ID. In particular,
\begin{enumerate}
  \item if $R$ is a field, every prime graded ideal of $\mathrm{KP}_R(\Lambda)$ is of the form $I_H$, where $\Lambda^0\setminus H$ is a maximal tail;
  \item if $R$ is a field and $\Lambda$ is strongly aperiodic, every prime ideal of $\mathrm{KP}_R(\Lambda)$ is of the form $I_H$, where $\Lambda^0\setminus H$ is a maximal tail.
\end{enumerate}
\end{proposition}

\begin{proof}
Let $I_{H}$ be a prime ideal of ${\mathrm{KP}_{R}(\Lambda)}$. Since ${\mathrm{KP}_{R}(\Lambda)}/{I_{H}}\cong \mathrm{KP}_{R}(\Lambda\setminus H)$ by Lemma \ref{lem4.5}, $\mathrm{KP}_{R}(\Lambda\setminus H)$ is a prime ring. Therefore, Theorem \ref{thm4.5} implies that $R$ is an ID and $\Lambda^{0}\setminus H$ satisfies Condition (MT3). Clearly, $\Lambda^0\setminus H$ also satisfies Conditions (MT1) and (MT2) because  $H$ is hereditary and saturated. Hence, $\Lambda^0\setminus H$ is a maximal tail.

Conversely, suppose that $\Lambda^{0}\setminus H$ is a maximal tail and $R$ is an ID. By Theorem \ref{thm4.5}, $\mathrm{KP}_{R}(\Lambda\setminus H)$ is a prime ring. Since ${\mathrm{KP}_{R}(\Lambda)}/ {I_{H}}\cong \mathrm{KP}_{R}(\Lambda\setminus H)$, we conclude that $I_{H}$ is a prime ideal of $\mathrm{KP}_{R}(\Lambda)$.
\end{proof}

In Theorem \ref{thm3.8}, we characterize the primitive Kumjian-Pask algebras when the underlying $k$-graphs have no sources. We may apply desourcifying method to obtain a same result for all row-finite locally convex $k$-graphs.

\begin{theorem}\label{thm4.7}
Let $\Lambda$ be a row-finite locally convex $k$-graph and $R$ be a unital commutative ring. Then $\mathrm{KP}_{R}(\Lambda)$ is primitive if and only if
\begin{enumerate}
\item[1)] $\Lambda^{0}$ satisfies Condition (MT3),
\item[2)] $\Lambda$ is aperiodic, and
\item[3)] $R$ is a field.
\end{enumerate}
\end{theorem}

\begin{proof}
Let $\mathrm{KP}_{R}(\Lambda)$ be a primitive ring. By \cite[Corollary 7.5]{clark}, there is surjective Morita context between $\mathrm{KP}_{R}(\Lambda)$ and $\mathrm{KP}_{R}(\tilde{\Lambda})$. As the primitivity is preserved under surjective Morita contexts, $\mathrm{KP}_{R}(\tilde{\Lambda})$ is also a primitive ring. Theorem \ref{thm3.8} implies that $\tilde{\Lambda}^{0}$ satisfies Condition (MT3), $\tilde{\Lambda}$ is aperiodic and $R$ is a field. Therefore, $\Lambda^{0}$ satisfies Condition (MT3) and $\Lambda$ is aperiodic by Lemma \ref{lem4.3}.

Conversely, let the above three conditions hold. By Lemma \ref{lem4.3}, $\tilde{\Lambda}^{0}$ satisfies Condition (MT3) and $\tilde{\Lambda}$ is aperiodic. So, $\mathrm{KP}_{R}(\tilde{\Lambda})$ is a primitive ring. The Morita equivalence between $\mathrm{KP}_{R}(\tilde{\Lambda})$ and $\mathrm{KP}_{R}(\Lambda)$ implies that $\mathrm{KP}_{R}(\Lambda)$ is also a primitive ring.
\end{proof}

\begin{proposition}\label{prop4.8}
Let $\Lambda$ be a row-finite locally convex $k$-graph and $R$ be a field. A graded ideal $I_{H}$ of $\mathrm{KP}_R(\Lambda)$ is primitive if and only if the $k$-graph $\Lambda\setminus H$ is aperiodic and $\Lambda^{0}\setminus H$ is a maximal tail. In particular, if $\Lambda$ is strongly aperiodic, then every primitive ideal of $\mathrm{KP}_R(\Lambda)$ is of the form $I_H$, where $\Lambda^0\setminus H$ is a maximal tail.
\end{proposition}

\begin{proof}
Let $I_{H}$ be a primitive ideal of ${\mathrm{KP}_{R}(\Lambda)}$. Since Lemma \ref{lem4.5} yields that ${\mathrm{KP}_{R}(\Lambda)}/{I_{H}}\cong \mathrm{KP}_{R}(\Lambda\setminus H)$ as rings, $\mathrm{KP}_{R}(\Lambda\setminus H)$ is a primitive ring. Therefore, Theorem \ref{thm4.7} implies that $\Lambda^{0}\setminus H$ satisfies Condition (MT3) and $\Lambda\setminus H$ is aperiodic.

Conversely, let $\Lambda^{0}\setminus H$ satisfy Condition (MT3) and $\Lambda\setminus H$ be aperiodic. By Theorem \ref{thm4.7}, the Kumjian-Pask algebra $\mathrm{KP}_{R}(\Lambda\setminus H)$ is a primitive ring. Since ${\mathrm{KP}_{R}(\Lambda)}/{I_{H}}\cong \mathrm{KP}_{R}(\Lambda\setminus H)$, we conclude that $I_{H}$ is a primitive ideal of $\mathrm{KP}_{R}(\Lambda)$.
\end{proof}

Considering Propositions \ref{prop4.6} and \ref{prop4.8}, we have the following result.

\begin{corollary}
Let $R$ be a field and $\Lambda$ be a row-finite locally convex $k$-graph that is strongly aperiodic. Then every prime ideal of $\mathrm{KP}_{R}(\Lambda)$ is primitive and vice versa.
\end{corollary}



\begin{thebibliography}{99}

\bibitem{abr1}G. Abrams and G. Aranda Pino, \emph{The Leavitt path algebras of a graph}, J. Algebra {\bf293} (2005), 319-334.

\bibitem{abr2}---------, \emph{The Leavitt path algebras of arbitrary graphs}, Houston J. Math. {\bf34} (2008), 323-442.

\bibitem{abrams} G. Abrams, J. P. Bell, and K. M. Rangaswamy, \emph{On prime non primitivevon Neumann regular algebras}, Trans. Amer. Math. Soc. {\bf 366} (2014), 2375-2392.

\bibitem{ara}P. Ara, M.A. Moreno, and E. Pardo, \emph{Nonstable K-theory for graph algebras}, Algebra Represent. Theory {\bf10} (2007), 157-178.

\bibitem{aranda} G. Aranda Pino, J. Clark, A. an Huef, and I. Raeburn,  \emph{Kumjian-Pask algebras of higher rank graphs}, Trans. Amer. Math. Soc. {\bf 365} (2013), 3613-3641.

\bibitem{ara3} G. Aranda Pino, E. Pardo, and M. Siles Molina, \emph{Prime spectrum and primitive Leavitt path algebras}, Indiana Univ. Math. J. {\bf58}(2) (2009), 869-890.

\bibitem{clark}
L.O. Clark, C. Flynn, and A. an Huef, \emph{Kumjian-Pask algebras of locally convex higher-rank graphs}, J. Algebra {\bf 399} (2014), 445-474.

\bibitem{Far}
C. Farthing, \emph{Removing sources from higher rank graphs}, J. Operator Theory. {\bf 60} (2008), 165-198.

\bibitem{forman}
E. Formanek, \emph{Group rings of free products are primitive}, J. Algebra. {\bf 59} (1979), 395-398.

\bibitem{kang}
S. Kang and D. Pask, \emph{Aperiodicity and primitive ideals of row-finite $k$-graphs},  Internat. J. Math. {\bf 25}(3) (2014), 1450022 (25 pages).

\bibitem{kumj00}
A. Kumjian and D. Pask, \emph{Higher rank graph $C^*$-algebras}, New York J. Math. {\bf 6} (2000), 1-20.

\bibitem{lan}
C. Lanski, R. Resco, and L. Small, \emph{On the primitivity of prime rings}, J. Algebra. {\bf 59}(2) (1979), 395-398.

\bibitem{lark12}
 H. Larki, \emph{Ideal structure of Leavitt path algebras with coefficients in a unital commutative ring}, Comm. Algebra {\bf 43}(12), (2015) 5031-5058.

\bibitem{lea}
W.G. Leavitt, \emph{The module type of a ring}, Trans. Amer. Math. Soc. {\bf 103} (1962), 113-130.

\bibitem{nas}
C. N$\check{\mathrm{a}}$st$\check{\mathrm{a}}$sescu and F. van Oystaeyen, \emph{Graded ring theory}, North-Holland, Amesterdam, 1982.

\bibitem{raeburn} I. Raeburn, A. Sims, and T. Yeend, \emph{Higher rank graphs and their $C^*$-algebras}, Proc. E dinb. Math. Soc. {\bf 46} (2003), 99-115.

\bibitem{ran}K.M. Rangaswamy, \emph{The theory of prime ideals of Leavitt path algebras over arbitrary graphs}, J. Algebra {\bf 375} (2013), 73 - 96.

\bibitem{Robert} D. Robertson and A. Sims, \emph{Simplicity of $C^*$-algebras associated to row-finite locally convex higher-rank graphs}, Israel J. Math. {\bf 172} (2009), 171-192.

\bibitem{steg1} G. Robertson and T. Steger, \emph{Affine buildings, titing systems and higher rank Cuntz- Krieger algebras }, J. Reine Angew. Math {\bf 513} (1999), 115-144.

\bibitem{tom1} M. Tomforde, \emph{Leavitt path algebras with coefficients in a commutative ring}, J. Pure Appl. Algebra {\bf 215} (2011), 471-484.

\end{thebibliography}
\end{document}